\documentclass[twoside,10pt]{amsart}

\usepackage{amssymb}
\usepackage{amsmath}

\usepackage{ifpdf}
\ifpdf
  \usepackage[colorlinks,linkcolor=red,anchorcolor=blue,citecolor=green]{hyperref}
\else
\fi

\numberwithin{equation}{section}

\newtheorem{theorem}{Theorem}[section]

\newtheorem{claim}[theorem]{Claim}

 \newtheorem{conjecture}[theorem]{Conjecture}
\newtheorem{corollary}[theorem]{Corollary}

\newtheorem{lemma}[theorem]{Lemma}

\newtheorem*{thmRRT}{Rainbow Ramsey Theorem}
\newtheorem*{thmRT}{Ramsey's Theorem}

\theoremstyle{definition}
\newtheorem{definition}[theorem]{Definition}

\newcommand{\<}{\langle}
\def\uh{\upharpoonright}
\renewcommand{\>}{\rangle}

\newcommand{\low}{\operatorname{low}}

\newcommand{\RCA}{\operatorname{RCA}_0}
\newcommand{\ACA}{\operatorname{ACA}_0}

\newcommand{\WKL}{\operatorname{WKL}_0}

\newcommand{\RT}{\operatorname{RT}}
\newcommand{\COH}{\operatorname{COH}}
\newcommand{\SRT}{\operatorname{SRT}}

\newcommand{\ADS}{\operatorname{ADS}}
\newcommand{\CADS}{\operatorname{CADS}}
\newcommand{\SADS}{\operatorname{SADS}}
\newcommand{\CAC}{\operatorname{CAC}}
\newcommand{\CCAC}{\operatorname{CCAC}}
\newcommand{\SCAC}{\operatorname{SCAC}}

\newcommand{\RRT}{\operatorname{RRT}}

\begin{document}

\title[Rainbow Ramsey Theorem for Triples]{Rainbow Ramsey Theorem for Triples Is Strictly Weaker Than the Arithmetical Comprehension Axiom}


\author{Wei Wang}

\thanks{The author thanks Chitat Chong and Yue Yang for sharing insights on Ramsey theory and reverse mathematics. He also thanks Hirschfeldt and Slaman for their inspiring lectures on Seetapun's theorem, IMS of NUS and the John Templeton Foundation for organizing and funding his participation in a series of logic programs, and the referees for comments and corrections. This research was partially supported by NSF Grant 11001281 of China and an NCET grant from the Ministry of Education of China.}

\address{Institute of Logic and Cognition and Department of Philosophy, Sun Yat-sen University, 135 Xingang Xi Road, Guangzhou 510275, P.R. China}
\email{wwang.cn@gmail.com}

\begin{abstract}
We prove that $\RCA + \RRT^3_2 \not\vdash \ACA$ where $\RRT^3_2$ is the Rainbow Ramsey Theorem for $2$-bounded colorings of triples. This reverse mathematical result is based on a cone avoidance theorem, that every $2$-bounded coloring of pairs admits a cone-avoiding infinite rainbow, regardless of the complexity of the given coloring. We also apply the proof of the cone avoidance theorem to the question whether $\RCA + \RRT^4_2 \vdash \ACA$ and obtain some partial answer.
\end{abstract}

\maketitle

\section{Introduction}

For computability theorists, Ramsey's theorem has been attractive for decades since Specker's work \cite{Specker:1971.Ramsey}. In this pioneering work, Specker showed that a computable $2$-coloring of pairs may admit no computable infinite homogeneous set. Let us recall some concepts from Ramsey theory. We use $[X]^n$ to denote the set of $n$-element subsets of $X$ where $n \leq \omega$; a function $f: [\omega]^n \to k$ is also called a \emph{$k$-coloring}, and a set $X$ is \emph{$f$-homogeneous} if $f$ is constant on $[X]^n$.

\begin{thmRT}
If $n, k \in \omega$ and $f$ is a $k$-coloring of $[\omega]^n$, then there exists an infinite $f$-homogeneous set.
\end{thmRT}

The instance of Ramsey's theorem for specific $n$ and $k$ is denoted by $\RT^n_k$. As a consequence of Specker's work, $\RCA \not\vdash \RT^2_2$. Here $\RCA$ denotes the Recursive Comprehension Axiom ($\RCA$), the weakest member of the \emph{big five} subsystems of second order arithmetic (see \cite{Simpson:1999.SOSOA}), and a base system for most of reverse mathematics. In this article, we shall also take $\RCA$ as a base system and always assume $\RCA$ without explicit reference.

Later, Jockusch \cite{Jockusch:1972.Ramsey} proved a series of interesting results concerning complexity of homogeneous sets in terms of arithmetic hierarchy. Moreover, Jockusch constructed a computable $2$-coloring of triples, for which every infinite homogeneous set computes the halting problem. In terms of reverse mathematics, Ramsey's theorem for triples ($\RT^3_2$) is equivalent to the Arithmetical Comprehension Axiom ($\ACA$), another member of the big five subsystems.

Jockusch's work left opened whether $\RT^2_2$ implies $\ACA$. This gap was overcame by Seetapun \cite{Seetapun.Slaman:1995.Ramsey}, who showed that $\ACA$ is strictly stronger than $\RT^2_2$. In his celebrated proof, Seetapun imposed some complexity conditions on Mathias forcing and got a subset of forcing conditions. In order to prove density lemmas for this subset, he exploited a cone avoidance theorem for $\Pi^0_1$ classes by Jockusch and Soare \cite{Jockusch.Soare:1972.TAMS}, which reflects the power of $\Pi^0_1$ classes in controlling complexity. From a computability theoretic point of view, Seetapun's proof shed deep insight into the classical proof of $\RT^2_2$.

Seetapun's proof was analyzed in Cholak, Jockusch and Slaman \cite{Cholak.Jockusch.ea:2001.Ramsey}, where two consequences of $\RT^2_2$ were introduced, namely $\COH$ and $\SRT^2_2$.

\begin{definition}
(1) Let $\vec{R} = (R_n: n \in \omega)$ be a sequence of subsets of $\omega$. An infinite set $X \subseteq \omega$ is \emph{$\vec{R}$-cohesive}, if and only if for each $n$ either $X \cap R_n$ or $X - R_n$ is finite.

$\COH$ is the assertion that there exists an $\vec{R}$-cohesive set for every $\vec{R}$.

(2) A $k$-coloring $f: [\omega]^2 \to k$ is \emph{stable} if and only if for every $x$ there exist $i < k$ and $y$ such that $f(x,z) = i$ for all $z > y$.

$\SRT^2_k$ is the assertion that there exists an infinite homogeneous set for every stable $k$-coloring.
\end{definition}

Cholak, Jockusch and Slaman \cite{Cholak.Jockusch.ea:2001.Ramsey} and Mileti \cite{Mileti:2004.thesis} showed that $\RT^2_2$ is equivalent to $\COH + \SRT^2_2$. Moreover, Cholak, Jockusch and Slaman obtained a sharp bound on complexity of homogeneous sets for coloring of pairs, and with this sharp bound they built an $\omega$-model of $\RT^2_2$ containing only $\low_2$ sets ($X$ is $\low_2$ if and only if $X'' \equiv_T \emptyset''$). The decomposition of $\RT^2_2$ into $\COH$ and $\SRT^2_2$ has became a paradigm of analysis of combinatorial principles relating to $\RT^2_2$. For example, in a similar vein Hirschfeldt and Shore \cite{Hirschfeldt.Shore:2007} decomposed the Ascending-Descending-Sequence principle ($\ADS$) into $\CADS$ and $\SADS$, and also the Chain-Antichain principle ($\CAC$) into $\CCAC$ and $\SCAC$; and respectively, they proved that $\CADS \not\vdash \SADS$, $\SADS \not\vdash \CADS$, $\CCAC \not\vdash \SCAC$ and $\SCAC \not\vdash \CCAC$.

Another interesting aspect of this analysis is that it gives a new proof of $\RT^2_2$. Actually, computability theoretic techniques are in need for controlling complexity of solutions to certain combinatorial principles, and usually lead to new insight into combinatorial proofs or even new proofs. A recent example is a joint work of Csima and Mileti on Rainbow Ramsey Theorems \cite{Csima.Mileti:2009.rainbow}.

\begin{definition}
A function $f: [\omega]^k \to \omega$ is a \emph{$b$-bounded coloring} if $|f^{-1}(c)| \leq b$ for each $c$. $X \subseteq \omega$ is an \emph{$f$-rainbow} if $f \uh [X]^k$ is injective.
\end{definition}

\begin{thmRRT}
$(\RRT^k_b)$ For each $b$-bounded $f: [\omega]^k \to \omega$ there exists an infinite $f$-rainbow.
\end{thmRRT}

A well known proof of $\RRT^k_b$ is by Galvin. For each $b$-bounded $f: [\omega]^k \to \omega$, Galvin defined a dual coloring $g: [\omega]^k \to b$ such that each $g$-homogeneous set is an $f$-rainbow, and then apply $\RT^k_b$ (see \cite{Csima.Mileti:2009.rainbow}). Csima and Mileti obtained a new proof of $\RRT^2_2$. Applying results from algorithmic randomness, Csima and Mileti showed that if $R$ is $2$-random relative to $Z$ and $f \leq_T Z$ is $2$-bounded then $R$ computes an infinite $f$-rainbow. From this, they deduced several reverse mathematical consequences, like $\RRT^2_2 \not\vdash \RT^2_2$, $\RRT^2_2 \not\vdash \WKL$, etc.

In this paper, we shall present some further investigations of $\RRT^k_b$.

In \S \ref{sect:triples}, we present the main result that $\RRT^3_2 \not\vdash \ACA$ (Theorem \ref{thm:RRT3-ACA}). As $\RT^3_2 \vdash \ACA$ by Jockusch \cite{Jockusch:1972.Ramsey}, this sounds a little surprising. It then follows from Jocksuch's result that $\RRT^3_2$ is strictly weaker than $\RT^3_2$. As one may expect, Theorem \ref{thm:RRT3-ACA} is based on some cone avoidance result (Lemma \ref{lem:RRT3-cone-avoidance}). By an application of a cone avoidance result for $\COH$, we reduce this cone avoidance for colorings of triples to a strong cone avoidance theorem for $2$-bounded colorings of pairs (Theorem \ref{thm:rainbow-for-pairs}). Theorem \ref{thm:rainbow-for-pairs} is strong, in that it gives us rainbows of \emph{low complexity} for colorings of \emph{arbitrary complexity}. We shall see in \S \ref{sect:pairs}, that the strength of Theorem \ref{thm:rainbow-for-pairs} actually comes from some hidden strength of Seetapun's cone-avoidance theorem for $\RT^2_2$, which was discovered by Dzhafarov and 
Jockusch \cite{Dzhafarov.Jockusch:2009}. The proof of Theorem \ref{thm:rainbow-for-pairs} combines measure theoretic argument from Csima and Mileti \cite{Csima.Mileti:2009.rainbow} and Mathias forcing in Seetapun's style, and also inductive applications of infinite pigeonhole principle. From a combinatorial viewpoint, the application of $\COH$ and inductive applications of pigeonhole principle together amount to inductive applications of $\RT^2_2$, and such applications of $\RT^2_2$ dispense the need of $\RT^3_2$ in building rainbows for colorings of triples.

In \S \ref{sect:quadruples}, we apply the method in \S \ref{sect:pairs} and \S \ref{sect:triples} to obtain some partial cone avoidance result for $2$-bounded colorings of quadruples. In \S \ref{sect:q}, we conclude this paper with a conjecture.

Before the presentations of results and proofs, we introduce some notions.

A \emph{model} of second order arithmetic is a pair $(M, \mathcal{S})$ where $M$ is a first order model of arithmetic and $\mathcal{S}$ is a subset of the powerset of $M$. An \emph{$\omega$-model} is a model $(M, \mathcal{S})$ with $M = \omega$. The least $\omega$-model of $\RCA$ is $(\omega, \mathcal{R})$ where $\mathcal{R}$ is the set of all computable sets. If $\mathcal{M} = (\omega, \mathcal{S})$ is a model and $X \subseteq \omega$, then $\mathcal{M}[X] = (\omega, \mathcal{S}[X])$, where $\mathcal{S}[X] = \{Z: \exists Y \in \mathcal{S}(Z \leq_T X \oplus Y)\}$. In addition, if $\mathcal{M} \models \RCA$ then $\mathcal{M}[X] \models \RCA$ for all $X$.

A \emph{tree} $T$ is a subset of $\omega^{< \omega}$ closed under initial segments. If $T$ is a tree, then $[T]$ denote the set of infinite sequences whose finite initial segments are always in $T$. A tree $T$ is \emph{$X$-computably bounded} if there exists an $X$-computable function $h$ such that $T \cap \omega^n \subseteq D_{h(n)}$ for each $n$, where $D_i$ is the $i$-th finite subset of $\omega^{< \omega}$ under some fixed computable coding.

Working with Ramsey-like combinatorial principles, we identify $[X]^{< \omega}$ with the set of strictly increasing sequences in $X^{< \omega}$ and $[X]^\omega$ with the set of infinite strictly increasing sequences from $X$. We use lower case Greek letters $\sigma, \tau, \ldots$ for elements of $[\omega]^{<\omega}$ and $\omega^{<\omega}$, and write $\sigma \tau$ for concatenation of $\sigma$ and $\tau$. Under the above convention, we may use concatenations for unions of finite sets, e.g., $\sigma \tau = \sigma \cup \tau$ and $\sigma \{x\} = \sigma \cup \{x\}$. We fix a computable bijection $\<\ldots\>: [\omega]^{<\omega} \to \omega$, such that
$$
    \<x_0,\ldots,x_{n-1}\> < \<y_0,\ldots,y_{n-1}\> \leftrightarrow \exists i < n(x_{i} < y_{i} \wedge \forall j \in (i,n)(x_j = y_j)).
$$

For computations using finite oracles, we write $\Phi_e(\sigma; x) \downarrow$ if $\Phi_e(\sigma; x)$ converges in no more than $|\sigma|$ steps. For $Z \subseteq \omega$ and $\sigma \in [\omega]^{<\omega}$, we write $Z \oplus \sigma$ for $(Z \uh |\sigma|) \oplus \sigma$. So, $\Phi_e(Z \oplus \sigma; x) \downarrow$ means $\Phi_e((Z \uh |\sigma|) \oplus \sigma; x) \downarrow$, etc.

Below, we summarize some useful results from Seetapun \cite{Seetapun.Slaman:1995.Ramsey} and Cholak, Jockusch and Slaman \cite{Cholak.Jockusch.ea:2001.Ramsey}, which will be called the \emph{cone avoidance} of $\RT^2_2, \COH$ or $\SRT^2_n$.

\begin{theorem}\label{thm:SS-CJS}
Suppose that $W \not\leq_T Z$.
\begin{enumerate}
    \item If $f: [\omega]^2 \to 2$ is $Z$-computable then there exists an infinite $f$-homogeneous $X$ such that $W \not\leq_T Z \oplus X$.
    \item If $\vec{R} = (R_n: n \in \omega)$ is uniformly $Z$-computable then there exists an infinite $\vec{R}$-cohesive $X$ such that $W \not\leq_T Z \oplus X$.
    \item If $g: \omega \to n$ is $Z'$-computable then there exist $k < n$ and $X \in [g^{-1}(k)]^\omega$ such that $W \not\leq_T Z \oplus X$.
\end{enumerate}
\end{theorem}

Dzhafarov and Jockusch \cite{Dzhafarov.Jockusch:2009} observed that a Seetapun-style Mathias forcing could work for Theorem \ref{thm:SS-CJS}(3) without the computability condition on $g$. Alternatively, we can remove this condition by an application of Theorem \ref{thm:SS-CJS}(3) itself, in a way similar to \cite[\S 2.4]{Seetapun.Slaman:1995.Ramsey}.

\begin{corollary}[Lemma 5.2(i) in \cite{Dzhafarov.Jockusch:2009}] \label{cor:Seetapun}
Suppose that $W \not\leq_T Z$ and $g: \omega \to n$ is any finite partition of $\omega$. Then there exist $k < n$ and $X \in [g^{-1}(k)]^\omega$ with $W \not\leq_T Z \oplus X$.
\end{corollary}

\begin{proof}
Following the proof of Friedberg's Jump Inversion Theorem (see \cite[VI.3]{Soare:1987.book}), we can build $Y$ such that $W \not\leq_T Y \oplus Z$ and $g \leq_T (Y \oplus Z)'$. Now the desired $k$ and $X$ can be obtained from Theorem \ref{thm:SS-CJS}(3).
\end{proof}

We call Corollary \ref{cor:Seetapun} as another \emph{Seetapun's cone avoidance} for infinite pigeonhole principle. The strength of the above corollary is the source of the strength of Theorem \ref{thm:rainbow-for-pairs}. Furthermore, one should note that Corollary \ref{cor:Seetapun} could be applied to remove the computability condition on $\vec{R}$ in Theorem \ref{thm:SS-CJS}(2). As we do not need such a strong result for $\COH$ here, we leave this as an exercise for practising Mathias forcing.

We shall need the following theorem by Jockusch and Soare \cite{Jockusch.Soare:1972.TAMS}.

\begin{theorem}\label{thm:Jockusch-Soare}
If $W \not\leq_T Z$ and $T \leq_T Z$ is an infinite $Z$-computably bounded tree then there exists $X \in [T]$ with $W \not\leq_T Z \oplus X$.
\end{theorem}

We refer readers to Simpson's book \cite{Simpson:1999.SOSOA} for more notions of reverse mathematics and Soare's book \cite{Soare:1987.book} for computability theoretic notions.

\section{Rainbows for Colorings of Pairs}\label{sect:pairs}

In this section, we prove the following cone avoidance theorem for $2$-bounded colorings of pairs. Note that there is no computability theoretic condition on $f$. In this sense, the following theorem is an analogous of Seetapun's cone avoidance for infinite pigeonhole principle (Corollary \ref{cor:Seetapun}). Actually, the proof needs inductively applications of Corollary \ref{cor:Seetapun}.

\begin{theorem}\label{thm:rainbow-for-pairs}
Suppose that $W \not\leq_T Z$. Then every $2$-bounded $f: [\omega]^2 \to \omega$ admits an infinite rainbow $H$ such that $W \not\leq_T H \oplus Z$.
\end{theorem}

Below we prove the above theorem. The proof consists of two main steps: firstly we pass from $\omega$ to an infinite tail rainbow $X$ such that $W \not\leq_T X \oplus Z$; then we prove the theorem for $f$ with $\omega$ being a tail rainbow. A set $X$ is a \emph{tail rainbow} for $f$, if $f(x_0,x_1) \neq f(y_0,y_1)$ for all $(x_0,x_1),(y_0,y_1) \in [X]^2$ with $x_1 \neq y_1$.

From \cite[Proposition 3.3]{Csima.Mileti:2009.rainbow}, we learn that the first step is easy for computable $f$. However, as now we are dealing with arbitrary colorings, we need some measure theoretic argument, which is essentially from \cite{Csima.Mileti:2009.rainbow}.

\begin{lemma}\label{lem:tail-rainbow-random}
If $f: [\omega]^2 \to \omega$ is $2$-bounded and $R$ is Martin-L\"{o}f random in $f$ then $f$ admits an infinite tail rainbow computable in $R$.

In particular, if $W \not\leq_T Z$ then $f$ admits an infinite tail rainbow $X$ with $W \not\leq_T X \oplus Z$.
\end{lemma}

\begin{proof}
Let $h(k) = k$ for $k \leq 2$. For $k > 2$, let
$$
    h(k) = h(k-1) + \min \{2^m: 2^m \geq 2^{k-1} \frac{(k-1)(k-2)}{2}\}.
$$
Let $S$ be the set of $\sigma \in [\omega]^{<\omega}$ such that $h(k) \leq \sigma(k) < h(k+1)$ for all $k < |\sigma|$, and let
$$
    T = \{\sigma \in S: \sigma \text{ is a tail rainbow for } f\}.
$$

Given any $\sigma \in [\omega]^{<\omega}$, we have
$$
    |f([\sigma]^2)| \leq |[\sigma]^2| = \frac{|\sigma|(|\sigma|-1)}{2}.
$$
By $2$-boundedness of $f$,
$$
    |\{x > \max \sigma: \exists w (w < x \wedge f(w,x) \in f([\sigma]^2))\}| \leq |f([\sigma]^2)| \leq \frac{|\sigma|(|\sigma|-1)}{2}.
$$
Thus, if $\sigma \in T$ then
$$
    |\{x: \sigma \{x\} \in T\}| \geq (1 - 2^{-|\sigma|}) |\{x: \sigma \{x\} \in S\}|.
$$
It follows that, for all $k$,
$$
    |T \cap [\omega]^k| \geq 2^{-1} |S \cap [\omega]^k|.
$$

By the definition of $S$, we can computably map $[S]$ onto $2^{\omega}$ by computably mapping $S$ to $2^{<\omega}$ as following: if $\sigma \in S$ is mapped to $\nu \in 2^{<\omega}$, then $\sigma \{x\} \in S$ is mapped to $\nu \xi$, so that $x$ is the $(\sum_{\xi(i) = 1} 2^i)$-th number $\geq h(|\sigma|)$. Under such mapping, $[T]$ is mapped onto a $\Pi^f_1$ subclass of $2^{\omega}$ with positive measure. By the corollary of Lemma 3 in Ku\v{c}era \cite{Kucera:85}, every $R$ which is Martin-L\"{o}f random in $f$, computes some $X \in [T]$ which is a tail $f$-rainbow.

On the other hand, if $W \not\leq_T Z$ then $\{Y \in 2^\omega: W \leq_T Y \oplus Z\}$ is null. So we can pick $R$ and $X \in [T]$ such that $R$ is Martin-L\"{o}f random in $f$, $X \leq_T R$ and $W \not\leq_T R \oplus Z$.
\end{proof}

With Lemma \ref{lem:tail-rainbow-random}, we proceed to the second step.

\begin{lemma}\label{lem:rainbow-for-pairs}
Suppose that $W \not\leq_T Z$. If $f$ is a $2$-bounded coloring of pairs with $\omega$ being a tail rainbow, then there exists an infinite $f$-rainbow $G$ such that $W \not\leq_T G \oplus Z$.
\end{lemma}

Below, we fix a coloring $f$ as in the above lemma and build a desired rainbow by Mathias forcing. The plan is as following:
\begin{enumerate}
    \item As $f$ is of arbitrary complexity, we can not directly consult $f$ in the construction. So, we define a $\Pi^0_1$ class (say $\mathcal{A}$) which captures $f$ in some sense.
    \item Instead of asking questions about $f$, we ask questions like: whether $\mathcal{A}$ contains some element satisfying certain $\Pi^Z_1$ property (say $\varphi$). Roughly, $\varphi(g)$ holds for a coloring $g \in \mathcal{A}$, if in a measure theoretic sense most $g$-rainbows do not contain splitting computations.
    \item If $\mathcal{A}$ does contain such an element, then we can pick a cone-avoiding $g \in \mathcal{A}$ by Jockusch-Soare's Theorem \ref{thm:Jockusch-Soare}. By a measure theoretic argument like \cite[\S 3]{Csima.Mileti:2009.rainbow}, we can obtain a cone-avoiding $g$-rainbow which contains no splitting computations. With this rainbow, we can extend a given Mathias condition to one which forces some $\Pi^Z_1$ statement (e.g., $\Phi_e(Z \oplus G; x) \uparrow$) and has desirable complexity.
    \item Otherwise, $f$ is a particular element of $\mathcal{A}$ which satisfies a $\Sigma^Z_1$ property ($\neg \varphi$). From this fact, we can extend a given Mathias condition in a finitary way to force a $\Sigma^Z_1$ statement, like $\Phi_e(Z \oplus G; x) \downarrow \neq W(x)$.
\end{enumerate}
People who are familiar with \cite{Seetapun.Slaman:1995.Ramsey} can find that the above plan is a variation of Seetapun's argument.

We begin with the definition of $\mathcal{A}$.

\begin{definition}\label{dfn:r-coloring}
Let $g: [\omega]^2 \to \omega$ be a $2$-bounded coloring. If
$$
    \forall (x_0,x_1) \in [\omega]^2(g(x_0, x_1) = \<\tilde{g}(x_0,x_1),x_1\>)
$$
where $\tilde{g}(x_0,x_1) = \min\{x \leq x_0: g(x,x_1) = g(x_0,x_1)\}$, then $g$ is \emph{normal}.

Let $\mathcal{A}$ be the set of all normal $2$-bounded colorings of $[\omega]^2$.
\end{definition}

Note that in this paper the notion of normal colorings is different from Csima and Mileti \cite{Csima.Mileti:2009.rainbow}. Let us briefly justify our terminology. If $g$ is a $2$-bounded coloring of pairs with $\omega$ being a tail rainbow, then we can define $\tilde{g}$ from $g$ as in the above definition and let
$$
    \bar{g}(x_0,x_1) = \<\tilde{g}(x_0,x_1), x_1\>.
$$
If $\bar{g}(x,x_1) = \bar{g}(y,x_1) = \<w,x_1\>$ and $x < y$, then $w = \tilde{g}(x,x_1) = \tilde{g}(y, x_1) \leq x$ and thus $g(w,x_1) = g(x,x_1) = g(y,x_1)$. By $2$-boundedness of $g$, $w = x$ and thus we have $2$-boundedness of $\bar{g}$ and $\bar{g} \in \mathcal{A}$. It also follows that $g$- and $\bar{g}$-rainbows coincide. On the other hand, suppose that $h_0, h_1 \in \mathcal{A}$ and $h_0(x,y) < h_1(x,y)$. Then $h_0(x,y) < h_1(x,y) \leq \<x,y\>$ and $h_0(x,y) = \<u,y\>$ for some $u < x$. As $h_0$ and $h_1$ are normal, $\{u,x,y\}$ is a rainbow for $h_1$ but not for $h_0$. So, every coloring like $f$ is equivalent to a unique element of $\mathcal{A}$, in the sense that they have same rainbows.

Hence, we can safely assume that $f \in \mathcal{A}$. Using some effective coding, $\mathcal{A}$ can be identified with a $\Pi^0_1$ subset of Cantor space.

We proceed to define a suitable subset of Mathias conditions. Let us begin from recalling standard Mathias forcing.

\begin{definition}\label{dfn:Mathias}
A \emph{Mathias condition} is a pair $(\sigma, X) \in [\omega]^{<\omega} \times [\omega]^{\omega}$ such that $\max \sigma < \min X$. For each Mathias condition $(\sigma, X)$, let
$$
    B(\sigma, X) = \{S \in [\omega]^\omega: \sigma \subseteq S \subseteq \sigma \cup X\}.
$$
Given two Mathias conditions $(\sigma, X)$ and $(\tau, Y)$, $(\tau, Y) \leq_M (\sigma, X)$ if and only if $B(\tau, Y) \subseteq B(\sigma, X)$.

Let $(\sigma, X)$ be a Mathias condition. We write $(\sigma, X)\Vdash \Phi_e(\dot{Z} \oplus \dot{G}) \neq \dot{W}$, if and only if $\Phi_e(Z \oplus G) \neq W$ for every $G \in B(\sigma, X)$.
\end{definition}

We need to find a descending sequence of Mathias condition $((\sigma_n, X_n): n \in \omega)$ such that $G = \bigcup_n \sigma_n$ is an infinite $f$-rainbow and $(\sigma_n, X_n) \Vdash \Phi_n(\dot{Z} \oplus \dot{G}) \neq \dot{W}$. For  forcing incomputability requirements, it is natural to require that $W \not\leq_T Z \oplus X_n$. For rainbow requirement, we need some further restriction on the infinite tails of Mathias conditions.

For $\sigma \in [\omega]^{<\omega}$ and $g \in \mathcal{A}$, let the set below collect all \emph{viable numbers}:
$$
    V(\sigma,g) = \{x > \max \sigma: \sigma \{x\} \text{ is a rainbow for }g\}.
$$
For a Mathias condition $(\sigma, X)$, let
$$
    \mathcal{A}_{\sigma,X} = \{g \in \mathcal{A}: X \subseteq V(\sigma, g)\} \in \Pi^X_1.
$$
To satisfy rainbow requirement, we should work with $(\sigma, X)$ having $f \in \mathcal{A}_{\sigma,X}$. Colorings in $\mathcal{A}_{\sigma,X}$ are considered as \emph{acceptable}.

\begin{definition}\label{dfn:Mathias-adm}
A Mathias condition $(\sigma, X)$ is \emph{admissible}, if and only if $f \in \mathcal{A}_{\sigma, X}$ and $W \not\leq_T Z \oplus X$.
\end{definition}

If (3) in the above plan holds for some $g \in \mathcal{A}_{\sigma,X}$, we need to pass from $X$ to a $g$-rainbow with desirable complexity. To this end, we define a tree of $g$-rainbows.

Let $b_0 = 1$. For each $l > 0$, let $b_l = \min\{2^b: 2^b \geq 2^{l+3}(|\sigma| + l)\}$.

\begin{definition}
Let $(\sigma, X)$ be a Mathias condition. If $g \in \mathcal{A}_{\sigma, X}$, then we associate to $(\sigma, X, g)$ a tree, denoted by $T(\sigma, X, g) \subseteq [\omega]^{<\omega}$, which is defined by induction as below:
\begin{enumerate}
    \item $\emptyset \in T(\sigma, X, g)$.
    \item Suppose that $\tau \in T(\sigma, X, g) \cap [\omega]^l$. If $x$ is among the least $b_l$ elements in $V(\sigma \tau, g) \cap X$ then $\tau \{x\} \in T(\sigma, X, g)$.
\end{enumerate}
\end{definition}

Some facts about $T(\sigma, X, g)$ follow easily from the definition:
\begin{itemize}
    \item[(T1)] $T(\sigma, X, g)$ is $(g \oplus X)$-computable uniformly in $(\sigma, X, g)$.
    \item[(T2)] $[T(\sigma, X, g)]$ is a compact subset of Baire Space $\omega^{\omega}$.
    \item[(T3)] $|T(\sigma, X, g) \cap [\omega]^l| \leq \bar{b}_l = \prod_{i < l} b_i$ for each $l > 0$.
    \item[(T4)] $\sigma \tau$ is a $g$-rainbow for each $\tau \in T(\sigma, X, g)$.
\end{itemize}

$T(\sigma,X,g)$ is similar to $T_f$ in \cite[\S 3]{Csima.Mileti:2009.rainbow}. But $T(\sigma,X,g)$ is compact, while $T_f$ in \cite[\S 3]{Csima.Mileti:2009.rainbow} is very wild. The compactness is needed in Lemma \ref{lem:cone-avoidance} below. We may prove that $T(\sigma, X, g)$ is bushy in a way similar to \cite[\S 3]{Csima.Mileti:2009.rainbow}. However, we shall need some different calculation.

\begin{lemma}\label{lem:rb-cntng}
Suppose that $(\sigma,X)$ is a Mathias condition and $g \in \mathcal{A}_{\sigma,X}$. For each $l$, let $y \in X$ be such that $y > \max \{\max \tau: \tau \in T(\sigma, X, g) \cap [\omega]^{l}\}$. Then
$$
    |\{\tau \in T(\sigma, X, g) \cap [\omega]^{l}: y \not\in V(\sigma \tau, g)\}| < \frac{\bar{b}_l}{4}.
$$
\end{lemma}

\begin{proof}
Let $T = T(\sigma, X, g)$. For non-empty $\tau \in [X]^{< \omega}$, let $x_{\tau} = \max \tau$ and $\tau^- = \tau - \{x_\tau\}$. For each $l$ and $y$, let
$$
    N_{l,y} = \{\tau \in T \cap [\omega]^{l}: y \not\in V(\sigma \tau, g)\}.
$$
We show by induction on $l$ that if $y \in X$ and $y > \max \{x_\tau: \tau \in T \cap [\omega]^{l}\}$ then
$$
    |N_{l,y}| < (\frac{1}{4} - \frac{1}{2^{l+2}}) \bar{b}_l.
$$

Let $y \in X$ be such that $y > \max \{x_\tau: \tau \in T \cap [\omega]^{l+1}\}$, and let
$$
    N_{l+1,y,0} = \{\tau \in T \cap [\omega]^{l+1}: \tau^- \in N_{l,y}\}.
$$
Then $|N_{l+1,y,0}| \leq |N_{l,y}| b_{l}$, by the definition of $T$. Let
$$
    N_{l+1,y,1} = N_{l+1,y} - N_{l+1,y,0}.
$$
Suppose that $\rho \in T \cap [\omega]^l - N_{l,y}$ and $\rho \{x\} \in N_{l+1,y,1}$. Then $y \in V(\sigma \rho, g) - V(\sigma \rho \{x\}, g)$. As both $\sigma \rho \{x\}$ and $\sigma\rho\{y\}$ are $g$-rainbows, $g(x,y) = g(u,v)$ for some $(u,v) \in [\sigma \rho \{x,y\}]^2 - \{(x,y)\}$. As $\omega$ is a tail $g$-rainbow, $v = y$. So, $g(x,y) = g(u,y)$ for some $u \in \sigma \rho$. As $g$ is $2$-bounded, there are at most $|\sigma| + l$ many $x$'s such that $\rho\{x\} \in N_{l+1,y,1}$. So
$$
    |N_{l+1, y, 1}| \leq (|T \cap [\omega]^l| - |N_{l,y}|) (|\sigma| + l) \leq (\bar{b}_l - |N_{l,y}|)(|\sigma| + l).
$$

Hence,
\begin{align*}
    |N_{l+1,y}| & \leq |N_{l,y}| b_l + (\bar{b}_l - |N_{l,y}|) (|\sigma| + l) \\
    & \leq |N_{l,y}| b_l + \frac{1}{2^{l+3}} b_l(\bar{b}_l - |N_{l,y}|) \\
    & < \frac{\bar{b}_{l+1}}{2^{l+3}} + (1 - \frac{1}{2^{l+3}}) (\frac{1}{4} - \frac{1}{2^{l+2}}) \bar{b}_{l+1} \\
    & < (\frac{1}{4} - \frac{1}{2^{l+3}}) \bar{b}_{l + 1}.
\end{align*}

This proves the lemma.
\end{proof}

\begin{lemma}\label{lem:bushy}
Suppose that $(\sigma,X)$ is a Mathias condition and $g \in \mathcal{A}_{\sigma, X}$. Then
$$
    |T(\sigma, X, g) \cap [\omega]^l| > \frac{3}{4} \bar{b}_l \text{ for all } l.
$$
\end{lemma}

\begin{proof}
For $l > 0$, let
$$
    M_l = \{\tau \in T(\sigma, X, g) \cap [\omega]^l: V(\sigma \tau, g) \text{ is infinite}\}.
$$
By the above lemma,
$$
    l > 0 \to |M_l| > \frac{3}{4} \bar{b}_l.
$$
This proves the lemma.
\end{proof}

By the above lemma and the definition of $T(\sigma,X,g)$, if $g \in \mathcal{A}_{\sigma,X}$ then we can $X \oplus g$-computably map $[T(\sigma,X,g)]$ onto some $[S] \subseteq 2^\omega$ with measure at least three quarters, where $S$ is a binary tree computably enumerable in $X \oplus g$. The mapping goes as following: if $\tau \in T(\sigma,X,g)$ is mapped to $\nu \in S$, then $\tau \{x\} \in T(\sigma,X,g)$ is mapped to $\nu \xi \in S$, so that $x$ is the $(\sum_{\xi(i) = 1} 2^{i})$-th element of $V(\sigma \tau, g) \cap X$. As $V(\sigma \tau,g) \cap X$ could be empty, the resulting $S$ can only be $\Sigma^{X \oplus g}_1$, although $T(\sigma,X,g)$ is $X \oplus g$-computable.

We are ready to extend an admissible condition to another that forces an incomputability requirement. To this end, we use splitting computations as usual.

\begin{definition}\label{dfn:e-split}
A finite sequence $\rho$ \emph{$(e, Z)$-splits over $\sigma$}, if and only if $\rho = \sigma \tau$ for some $\tau$ and there are different $\tau_0$ and $\tau_1$ such that both $\tau_0$ and $\tau_1$ are subsets of $\tau$ and $\Phi_e(Z \oplus (\sigma\tau_0); x) \downarrow \neq \Phi_e(Z \oplus (\sigma\tau_1); x) \downarrow$ for some $x$.

If $g \in \mathcal{A}_{\sigma, X}$ is such that
$$
    |\{\tau \in T(\sigma, X, g) \cap [\omega]^l: \sigma \tau \text{ $(e, Z)$-splits over } \sigma\}| < \frac{\bar{b}_l}{2}
$$
for all $l$, then \emph{$g$-rainbows $(e, Z)$-split over $(\sigma, X)$ with low probability}.
\end{definition}

It is a $\Pi^{Z \oplus X \oplus g}_1$ question whether $g$-rainbows $(e, Z)$-split over $(\sigma, X)$ with low probability. This complexity bound is due to $T(\sigma,X,g)$ being nicely bounded. This explains why we define $T(\sigma,X,g)$ in a way different from $T_f$ in \cite{Csima.Mileti:2009.rainbow}.

\begin{lemma}\label{lem:cone-avoidance}
Suppose that $(\sigma, X)$ is an admissible Mathias condition. For each $e$, there exists an admissible $(\tau, Y) \leq_M (\sigma, X)$ such that $(\tau, Y) \Vdash \Phi_e(\dot{Z} \oplus \dot{G}) \neq \dot{W}$.
\end{lemma}

\begin{proof}
Let $\mathcal{U}$ be the set of $g \in \mathcal{A}_{\sigma, X}$ such that $g$-rainbows $(e, Z)$-split over $(\sigma, X)$ with low probability. Clearly, $\mathcal{U}$ is $\Pi^{Z \oplus X}_1$. There are two cases.

\emph{Case 1:} $\mathcal{U} \neq \emptyset$. By Theorem \ref{thm:Jockusch-Soare} of Jockusch and Soare, there exists $g \in \mathcal{U}$ such that $W \not\leq_T Z \oplus g \oplus X$. Let
$$
    T = T(\sigma, X, g) \cap \{\tau \in [\omega]^{<\omega}: \sigma \tau \text{ does not $(e, Z)$-split over } \sigma\}.
$$
As $T(\sigma, X, g)$ is computable in $g \oplus X$, $T$ is computable in $Z \oplus X \oplus g$. By Lemma \ref{lem:bushy} and that $g$-rainbows $(e, Z)$-split over $(\sigma, X)$ with low probability,
$$
    \forall l (|T \cap [\omega]^l| \geq \frac{\bar{b}_l}{4}).
$$
Combining the above inequality and the remark following Lemma \ref{lem:bushy}, and by an application of the corollary of Lemma 3 in Ku\v{c}era \cite{Kucera:85}, if $R$ is 2-random in $Z \oplus g \oplus X$ then $Z \oplus g \oplus X \oplus R$ computes some $Y \in [T]$. As $W \not\leq_T Z \oplus g \oplus X$, there exist $Y$ and $R$ such that $R$ is 2-random in $Z \oplus g \oplus X$, $Y \leq_T Z \oplus g \oplus X \oplus R$, $Y \in [T]$ and $W \not\leq_T Z \oplus Y$. Note that each path of $T$ is a subset of $X$. So, $(\sigma, Y)$ is an admissible extension of $(\sigma, X)$.

It remains to show that $(\sigma, Y) \Vdash \Phi_e(\dot{Z} \oplus \dot{G}) \neq \dot{W}$. Suppose that for each $x$ there exists $\tau \in [Y]^{<\omega}$ such that $\Phi_e(Z \oplus (\sigma \tau); x) \downarrow$. Then there must be some $x$ such that $\Phi_e(Z \oplus (\sigma \tau); x) \neq W(x)$ whenever $\Phi_e(Z \oplus (\sigma \tau); x) \downarrow$ for $\tau \in [Y]^{<\omega}$, as $W \not\leq_T Z \oplus Y$ and $\sigma \cup Y$ does not $(e, Z)$-split over $\sigma$.

\emph{Case 2:} $\mathcal{U} = \emptyset$. In particular $f \not\in \mathcal{U}$. By the definition of $\mathcal{U}$, there exist some $l > 0$ and
$$
    \{\tau_i: i < b = \frac{\bar{b}_l}{2}\} \subseteq \{\tau \in T(\sigma, X, f) \cap [\omega]^l: \sigma \tau \text{ $(e, Z)$-splits over } \sigma\}.
$$

Let $m = \max \bigcup_{i < b} \tau_i$. By Lemma \ref{lem:rb-cntng}, for each $y \in X - [0, m]$ there exists $i < b$ such that $y \in V(\sigma \tau_i, f)$. Let $p(y)$ be the least such $i$. So, $p$ is a finite partition of $X - [0, m]$. By Seetapun's cone avoidance for infinite pigeonhole principle (Corollary \ref{cor:Seetapun}), there exist $i < b$ and $Y \in [X \cap p^{-1}(i)]^\omega$ such that $W \not\leq_T Z \oplus X \oplus Y$. It follows that $f \in \mathcal{A}_{\sigma \tau_i, Y}$. As $\sigma\tau_i$ $(e, Z)$-splits over $\sigma$, there are $x$, $\pi_0$ and $\pi_1$ such that $\sigma \subseteq \pi_j \subseteq \sigma \tau_i$ for $j < 2$ and $\Phi_e(Z \oplus \pi_0; x) \downarrow \neq \Phi_e(Z \oplus \pi_1; x) \downarrow$. Take $\tau = \pi_j$ such that $\Phi_e(Z \oplus \pi_j; x) \neq W(x)$.

So, $(\tau, Y) \leq_M (\sigma, X)$ is admissible and $(\tau, Y) \Vdash \Phi_e(\dot{Z} \oplus \dot{G}) \neq \dot{W}$.
\end{proof}

An argument similar to the second case in the above proof allows us to extend the finite heads of admissible Mathias conditions.

\begin{lemma}\label{lem:finite-ext}
If $(\sigma, X)$ is an admissible Mathias condition then there exists an admissible $(\tau, Y) \leq_M (\sigma, X)$ such that $\sigma \subset \tau$.
\end{lemma}

\begin{proof}
Let $l = |\sigma|$ and $x_0, \ldots, x_l$ be the first $l + 1$ elements of $X$ in ascending order. For each $y \in X \cap (x_l,\infty)$, let $p(y)$ be the least $i \leq l$ such that $y \in V(\sigma \{x_i\}, f)$. By the $2$-boundedness of $f$, $p$ is total on $X \cap (x_l,\infty)$. By Seetapun's cone avoidance for infinite pigeonhole principle again, there exist $i \leq l$ and $Y \in [X \cap p^{-1}(i)]^\omega$ such that $W \not\leq_T Z \oplus X \oplus Y$. Clearly, $(\sigma \{x_i\}, Y)$ is as desired.
\end{proof}

By Lemmata \ref{lem:cone-avoidance} and \ref{lem:finite-ext}, we can obtain a descending sequence of Mathias conditions $((\sigma_n, X_n): n \in \omega)$ such that $(\sigma_n,X_n) \Vdash \Phi_n(\dot{Z} \oplus \dot{G}) \neq \dot{W}$ and $\sigma_n \subset \sigma_{n+1}$ for all $n$. So, $G = \bigcup_n \sigma_n$ is a desired $f$-rainbow. This proves Lemma \ref{lem:rainbow-for-pairs}.

Theorem \ref{thm:rainbow-for-pairs} follows immediately.

\section{Rainbows for Colorings of Triples}\label{sect:triples}

In this section, we prove the main theorem.

\begin{theorem}\label{thm:RRT3-ACA}
$\RRT^3_2 \not\vdash \ACA$. Thus $\RRT^3_2$ is strictly weaker than $\RT^3_2$.
\end{theorem}

To this end, for each non-computable $W$ we build an $\omega$-model $\mathcal{M} = (\omega, \mathcal{S})$ such that $\mathcal{M} \models \RCA + \RRT^3_2$ and $W \not\in \mathcal{S}$. In particular, taking $W$ to be a non-computable arithmetic set, yields that $\mathcal{M} \not\models \ACA$.

The key is to find an $f$-rainbow $X$ with $W \not\leq_T Z \oplus X$, whenever $W \not\leq_T Z$ and $f: [\omega]^3 \to \omega$ is 2-bounded and $Z$-computable. If we achieve this, then we can inductively build a desired $\mathcal{S}$.

We state the above key step as a lemma.

\begin{lemma}\label{lem:RRT3-cone-avoidance}
If $W \not\leq_T Z$ and $Z$ computes a $2$-bounded $f: [\omega]^3 \to \omega$, then there exists an infinite $f$-rainbow $X$ such that $W \not\leq_T Z \oplus X$.
\end{lemma}

\begin{proof}
Let $W, Z$ and $f$ be as above. By \cite[Proposition 3.3]{Csima.Mileti:2009.rainbow}, we may assume that $\omega$ is a \emph{$1$-tail $f$-rainbow}, i.e., $f(u,v,w) \neq f(x,y,z)$ whenever $w \neq z$. Let
$$
    I = \{(u,v,x,y) \in [\omega]^2 \times [\omega]^2: \<u,v\> < \<x,y\>\}.
$$
Let $\bar{f}(x,y,z) = \min \{\<u,v\>: f(u,v,z) = f(x,y,z)\}$ for each $(x,y,z) \in [\omega]^3$, and let
$$
    R_{u,v,x,y} = \{s > y: \bar{f}(x,y,s) = \<u,v\>\}
$$
for $(u,v,x,y) \in I$. Note that $\bar{f}(x,y,z) \leq \<x,y\>$. By the cone avoidance of $\COH$, there exists $C$ such that $C$ is cohesive for $\vec{R} = (R_{u,v,x,y}: (u,v,x,y) \in I)$ and $W \not\leq_T Z \oplus C$.

For $(x,y) \in [C]^2$, let
$$
    \hat{f}(x,y) =
    \left\{
      \begin{array}{ll}
        \<u,v\>, & \<u,v\> = \lim_{s \in C} \bar{f}(x,y,s) \wedge (u,v) \in [C]^2;  \\
        \<x,y\>, & \hbox{otherwise.}
      \end{array}
    \right.
$$
By cohesiveness of $C$, $\hat{f}$ is well defined. Moreover, if $\hat{f}(x,y) = \<u,v\>$ and $(u,v) \in [C]^2$, then $\<u,v\> \leq \<x,y\>$ and $\bar{f}(x,y,s) = \<u,v\>$ for sufficiently large $s \in C$. Thus $f(x,y,s) = f(u,v,s)$ for sufficiently large $s \in C$. So, if $\hat{f}(x,y) = \hat{f}(x',y')$ then $f(x,y,s) = f(x',y',s)$ for sufficiently large $s \in C$. As $f$ is $2$-bounded, $\hat{f}$ is also $2$-bounded. By Theorem \ref{thm:rainbow-for-pairs}, there exists an infinite $\hat{f}$-rainbow $G$ such that $G \subseteq C$ and $W \not\leq_T Z \oplus G$.

We define a desired $X$ as a subset of $G$ by induction. Let $X_0 = \emptyset$. Suppose that $X_n \in [G]^{<\omega}$ is defined and $X_n$ is a rainbow for $f$.

\begin{claim}\label{clm}
For all sufficiently large $a \in G$, $X_n \cup \{a\}$ is a rainbow for $f$.
\end{claim}

\begin{proof}[Proof of Claim \ref{clm}]
As $X_n$ is finite and $G$ is $\vec{R}$-cohesive, there exists $s_0$ such that $\bar{f}(x,y,a) = \lim_{s \in C} \bar{f}(x,y,s)$ for all $a \in G \cap (s_0, \infty)$ and $(x,y) \in [X_n]^2$.

For $a \in G \cap (s_0, \infty)$, if there are $(x,y)$ and $(x',y')$ in $[X_n]^2$ such that $\<x,y\> < \<x',y'\>$ and $f(x,y,a) = f(x',y',a)$, then $\bar{f}(x',y',a) = \<x,y\>$ and $\hat{f}(x',y') = \<x,y\> = \hat{f}(x,y)$. But this is impossible, as $X_n \subset G$ is a rainbow for $\hat{f}$.

So, $f(x,y,a) \neq f(x',y',a)$ for distinct $(x,y)$ and $(x',y')$ in $[X_n]^2$ and all $a \in G \cap (s_0, \infty)$. As $X_n$ is an $f$-rainbow and $\omega$ is a $1$-tail $f$-rainbow, $X_n \cup \{a\}$ is an $f$-rainbow.
\end{proof}

By the above claim, let $a_n$ be the least $a \in G \cap (\max X_n, \infty)$ such that $X_n \cup \{a\}$ is a rainbow for $f$. Let $X_{n+1} = X_n \cup \{a_n\}$.

So, $X = \bigcup_n X_n$ is an infinite $f$-rainbow. As $X \leq_T Z \oplus G$, $W \not\leq_T Z \oplus X$.
\end{proof}

\begin{proof}[Proof of Theorem \ref{thm:RRT3-ACA}]
Fix a non-computable $W$. By inductive applications of Lemma \ref{lem:RRT3-cone-avoidance}, we can build a sequence $(\mathcal{M}_n: n < \omega)$ such that
\begin{enumerate}
    \item $\mathcal{M}_0$ is the least $\omega$-model of $\RCA$;
    \item Each $\mathcal{M}_{n + 1} = \mathcal{M}_n[X_n]$ for some $X_n$;
    \item For all $n > 0$, $W \not\leq_T \bigoplus_{i < n} X_i$;
    \item If $f: [\omega]^3 \to \omega$ is a $2$-bounded coloring in $\mathcal{M}_n$ then there exists some $m \geq n$ such that $X_m$ is an infinite $f$-rainbow;
\end{enumerate}
Then $\mathcal{M} = \bigcup_n \mathcal{M}_n$ is a model of $\RCA + \RRT^3_2$ and $W$ is not in $\mathcal{M}$.

In particular, taking $W$ to be the halting problem, yields that $\mathcal{M} \not\models \ACA$.
\end{proof}

Combining the above proof with the proof of Seetapun's Theorem \cite[Theorem 2.1]{Seetapun.Slaman:1995.Ramsey}, we get the following corollary.

\begin{corollary}
$\WKL + \RT^2_2 + \RRT^3_2 \not\vdash \ACA$.
\end{corollary}

\section{Tail Rainbows for Colorings of Quadruples}\label{sect:quadruples}

In this section, we present some partial answer to the question whether $\RRT^n_2 \vdash \ACA$ for some $n > 3$. We obtain a cone avoidance result that every computable $2$-bounded coloring of $[\omega]^4$ admits a cone-avoiding infinite rainbow-like set.

If $k \leq n$, $X$ and $f: [\omega]^{n+1} \to \omega$ are such that $f(x_0,x_1,\ldots,x_n) \neq f(y_0,y_1,\ldots,y_n)$ for $(x_0,x_1,\ldots,x_n)$ and $(y_0,y_1,\ldots,y_n)$ in $[X]^{n+1}$ with distinct $(x_{n-k+1},\ldots,x_n)$ and $(y_{n-k+1},\ldots,y_n)$, then we say that $X$ is a \emph{$k$-tail $f$-rainbow}. A \emph{tail $f$-rainbow} is just a $k$-tail $f$-rainbow for $k = n$.

\begin{theorem}\label{thm:quadruples}
Suppose that $W \not\leq_T Z$ and $f: [\omega]^4 \to \omega$ is a $2$-bounded coloring computable in $Z$. Then there exists an infinite tail $f$-rainbow $X$ such that $W \not\leq_T X \oplus Z$.
\end{theorem}

We follow the proof of Lemma \ref{lem:RRT3-cone-avoidance}. By \cite[Proposition 3.3]{Csima.Mileti:2009.rainbow} and the cone avoidance of $\COH$, it suffices to prove the following cone avoidance result, which is an analogous of Theorem \ref{thm:rainbow-for-pairs}.

\begin{lemma}\label{lem:RRT4-t-rainbow-c-a}
Suppose that $W \not\leq_T Z$ and $\hat{f}: [\omega]^3 \to \omega$ is $2$-bounded. Then there exists an infinite tail $\hat{f}$-rainbow $X$ such that $W \not\leq_T X \oplus Z$.
\end{lemma}

Fix $W, Z$ and $\hat{f}$ as in the above lemma. By an argument similar to that of Lemma \ref{lem:tail-rainbow-random}, we can assume that $\omega$ is a $1$-tail rainbow for $\hat{f}$.

To apply Mathias forcing as in \S \ref{sect:pairs}, we define a class of colorings. For a $2$-bounded $g: [\omega]^{3} \to \omega$, if
$$
    \forall (x_0,x_1,x_{2}) \in [\omega]^{3}(g(x_0,x_1,x_{2}) = \<\tilde{g}(x_0,x_1,x_2),x_2\>)
$$
where $\tilde{g}(x_0,x_1,x_2) = \min\{\<y_0,y_1\>: g(y_0,y_1,x_2) = g(x_0,x_1,x_2)\}$, then $g$ is \emph{semi-normal}. Let $\mathcal{B}$ be the set of all semi-normal $2$-bounded coloring of $[\omega]^{3}$. Clearly, $\omega$ is a $1$-tail rainbow for every $g \in \mathcal{B}$ and $\mathcal{B}$ can be identified as a $\Pi^0_1$ subset of Cantor space under some effective coding. Moreover, every $2$-bounded coloring of triples with $\omega$ being a $1$-tail rainbow is equivalent to some element of $\mathcal{B}$, in the sense of having same tail rainbows. So, we may assume that $\hat{f} \in \mathcal{B}$.

Fix $\sigma \in [\omega]^{< \omega}$ and $g \in \mathcal{B}$. Let $tV(\sigma, g)$ be the set of \emph{tail viable} numbers, i.e., $x \in tV(\sigma,g)$ if and only if $\sigma \{x\}$ is a tail $g$-rainbow. For a Mathias condition $(\sigma, X)$, let
$$
    \mathcal{B}_{\sigma, X} = \{g \in \mathcal{B}: X \subseteq tV(\sigma, g)\} \in \Pi^X_1.
$$

If $g \in \mathcal{B}_{\sigma, X}$ then we associate a tree $\tilde{T}(\sigma, X, g) \subseteq [X]^{<\omega}$ to $(\sigma, X, g)$. $\tilde{T}(\sigma, X, g)$ is defined by induction as following.
\begin{enumerate}
    \item $\emptyset \in \tilde{T}(\sigma, X, g)$,
    \item If $\tau \in \tilde{T}(\sigma, X, g)$ and $x$ is among the first $c_{|\sigma \tau|}$ many elements in $tV(\sigma \tau, g) \cap X \cap (\max \tau,\infty)$ where $c_k = \min\{2^c: 2^c \geq 2^{k+3} \binom{k}{2}$, then $\tau \{x\} \in \tilde{T}(\sigma, X, g)$.
\end{enumerate}
$\tilde{T}(\sigma, X, g)$ is similar to the trees in \S \ref{sect:pairs}. If $\tau \in \tilde{T}(\sigma,X,g)$ then $\sigma\tau$ is a tail $\hat{f}$-rainbow.

Fix $\tau \in \tilde{T}(\sigma,X,g)$ and $y \in X \cap tV(\sigma\tau,g)$. If $\tau\{x\} \in \tilde{T}(\sigma,X,g)$ and $y \in tV(\sigma\tau,g) - tV(\sigma\tau\{x\},g)$, then there are $w,u,v \in \sigma\tau$ such that $g(w,x,y) = g(u,v,y)$. By $2$-boundedness of $g$, for our fixed $y$ and $\tau$, there could be at most $\binom{|\sigma\tau|}{2}$ many $x$'s as above. This simple calculation allows us to establish something like Lemmata \ref{lem:rb-cntng} and \ref{lem:bushy} for $\tilde{T}(\sigma,X,g)$. So, with this new family of trees, we can construct an appropriate Mathias generic as in \S \ref{sect:pairs} and thus obtain a tail $\hat{f}$-rainbow $G$ such that $W \not\leq_T Z \oplus G$. The details are left to the reader.

Actually, we can slightly generalize Theorem \ref{thm:quadruples} to the following corollary.

\begin{corollary}
Suppose that $W \not\leq_T Z$ and $f: [\omega]^{n+4} \to \omega$ is $2$-bounded and $Z$-computable. Then there exists an infinite $3$-tail $f$-rainbow $X$ such that $W \not\leq_T X \oplus Z$.
\end{corollary}

\section{A Conjecture}\label{sect:q}

Inspired by \cite[Theorem 3.1]{Cholak.Jockusch.ea:2001.Ramsey}, we may wonder whether we can control triple jumps of rainbows for colorings of triples. This is recently confirmed by the author in an upcoming work: every computable $2$-bounded coloring of triples admits a $\low_3$ infinite rainbow $X$ (i.e., $X''' \equiv_T \emptyset'''$). As a consequence of this answer and the lower complexity bound in \cite[Theorem 2.5]{Csima.Mileti:2009.rainbow}, $\RRT^3_2 \not\vdash \RRT^4_2$. However, the proof for controlling triple jumps is complicated, and apparently hard to be adapted to override cone avoidance results here.

Finally I boldly conjecture the following.

\begin{conjecture}
For all $n$, $\RCA + \RRT^n_2 \not\vdash \ACA$.
\end{conjecture}

\bibliographystyle{plain}

\end{document}